\documentclass[A4paper,10pt]{amsart}
\usepackage[foot]{amsaddr}
\usepackage{amsmath} 
\usepackage{amssymb} 
\usepackage{
hyperref}
\usepackage{marginnote}
\usepackage{enumitem}
\usepackage{tikz}
\usepackage[normalem]{ulem}
\usepackage{centernot}

\newcommand{\vertiii}[1]{{\left\vert\kern-0.25ex\left\vert\kern-0.25ex\left\vert #1 
    \right\vert\kern-0.25ex\right\vert\kern-0.25ex\right\vert}}

\def\text{\textrm}

 \newtheorem{theorem}{Theorem}[section]
 
 \newtheorem{lemma}[theorem]{Lemma}
 \newtheorem{proposition}[theorem]{Proposition}
 \theoremstyle{definition}
 \newtheorem{definition}[theorem]{Definition}
 \theoremstyle{remark}
 \newtheorem{rem}[theorem]{Remark}
 
 \numberwithin{equation}{section}

\begin{document}

\title[Strong ISS for linear systems]{Strong input-to-state stability for infinite dimensional linear systems}

\date{}

\author[R.~Nabiullin]{Robert Nabiullin}
\address{Functional analysis group, School of Mathematics and Natural Sciences,
        University of Wuppertal, D-42119 Wuppertal, Germany} 
\email{nabiullin@math.uni-wuppertal.de}
\author[F.L.~Schwenninger]{Felix L.~Schwenninger}
\address{Department of Mathematics, Center for Optimization and Approximation, University of Hamburg, 
Bundesstra\ss e 55, D-20146 Hamburg, Germany}
\email{felix.schwenninger@uni-hamburg.de}
\thanks{RN is supported by Deutsche Forschungsgemeinschaft (Grant JA 735/12-1). FLS is supported by Deutsche Forschungsgemeinschaft (Grant RE 2917/4-1)}

\begin{abstract}
This paper deals with strong versions of input-to-state stability and integral input-to-state stability of infinite-dimensional linear systems with an unbounded input operator. We show that infinite-time admissibility with respect to inputs in an Orlicz space is a sufficient condition for a system to be strongly integral input-to-state stable but, unlike in the case of exponentially stable systems, not a necessary one.
\end{abstract}

 

\maketitle

\section{Introduction}
\label{intro}
{\it Input-to-state stability (ISS)} is a well-studied notion in the analysis of (robust) stability for nonlinear ODE systems and goes back to E.~Sontag \cite{Sontag89ISS}, see \cite{sontagISS} for an overview. The PDE case, however, has developed only in the recent past --- see \cite{JaLoRy08} for the first work in that area, \cite{DaM13,KK2016IEEETac,Krstic16,MaPr11,MiI14b}, and \cite{MiW17b} as well as the references therein for the state of the art --- and is still subject to ongoing research. One of the questions is, how variants of ISS known from  finite dimensions \cite{Sontag98} --- especially \textit{integral input-to-state stability (iISS)} --- relate in the infinite-dimensional setting. Surprisingly, even for linear systems this still not fully understood, see \cite{JNPS16}. With this contribution we aim to make a step towards clarifying that case. This will complement the findings in \cite{JNPS16}.

Let us in the following always consider linear control systems of the form 
\begin{equation}\label{eqn1}
\dot{x}(t)=Ax(t)+Bu(t),\quad  t\geq0,\quad x(0)=x_{0},
\end{equation}
where $A$ generates a $C_0$-semigroup  and $B$ is a possibly unbounded input operator. Input-to-state stability enables a way to jointly describe the stability of the mappings $x_{0}\mapsto x(t)$ and $u\mapsto x(t)$ for fixed $t>0$.
The specific ISS notion depends on the `norm' in which the functions $u$ are measured and in which sense the internal stability, that is, the stability of the semigroup, is understood. For the latter we will consider \textit{strong stability} rather than the more restrictive exponential stability, which is somewhat in contrast to the usual setting for (infinite-dimensional) ISS in the literature. This accounts for the notion of {\it strong ISS (sISS)}, which was recently introduced in \cite{MiW17b}.  There, it is also shown that for certain nonlinear systems, sISS is equivalent to the {\it strong asymptotic gain property} together with \textit{uniformly global stability}.

In this article, we show that \textit{strong integral input-to-state stability (siISS)} is implied by sISS with respect to an Orlicz space, Theorem \ref{thm:main2b}. As we are dealing with linear systems, the later is equivalent to \textit{infinite-time admissibility} with respect to some Orlicz space $E_\Phi$ together with strong stability of the semigroup. A link between (integral) ISS and Orlicz-space-admissibility was recently established by Jacob, Partington and the authors in \cite{JNPS16}. There the exponential stability of the semigroup played a significant role in proving that iISS is equivalent to ISS with respect to an Orlicz space, see \cite[Thm.~15]{JNPS16}. However, in the present paper we show that this equivalence is no longer true in the more general situation of siISS and sISS, Theorem \ref{thm:counterex}.

 It is known (and easy to see) that, for linear systems \eqref{eqn1}, iISS
implies ISS with respect to $L^{\infty}$. Whether equivalence holds in general, is still an open problem (and connected to a fundamental open question raised by G.~Weiss on the continuity of mild solutions, see \cite{JNPS16,Weiss89ii}). In certain situations this is known to be true, like, for instance, for parabolic diagonal systems \cite{JNPS16} or, in some sense more general, for
analytic semigroups on Hilbert spaces that are similar to a contraction semigroup  \cite{JacSchwZwa} --- in both cases for finite-dimensional input spaces. Furthermore, equivalence also holds in the case of bounded operators $B$, \cite{MiI14b}.
For sISS and siISS, the situation is different. We show that the implication siISS $\implies$ sISS with respect to $L^{\infty}$ fails in general, even if $B$ is bounded. Hence without exponential stability, strong integral ISS and sISS (with respect to $L^{\infty}$) cannot be equivalent, Theorem \ref{thm:counterex}. 

The reason for the failure of the above equivalences lies in the fact that we have to distinguish between ``finite-time admissibility'' and the stronger ``infinite-time admissibility'' here. This is different to the relation between ISS and iISS. This also shows that for general input-to-state stability of linear systems, the stability concepts of the mappings $x_{0}\mapsto x(t)$ and $u\mapsto x(t)$ may not be viewed separately.

In Section \ref{sec:2} we introduce the class of linear systems and the required stability
concepts we are
dealing with. Section \ref{sec:3} contains our main results.

\section{Admissibility and stability of infinite-dimensional systems}\label{sec:2}
Throughout the whole article we consider linear systems $\Sigma(A,B)$ given by \eqref{eqn1},
where $A$ is the generator of a $C_0$-semigroup $(T(t))_{t\ge 0}$ on a Banach space $X$, $U$ is another Banach space and $B \in \mathcal L(U,X_{-1})$. The space $X_{-1}$ is defined to be the completion of $X$ with respect to the norm given by $\|x\|_{-1} := \|(\lambda I-A)^{-1}x\|$, where $\lambda$ is some element of $\rho(A)$, the resolvent set of $A$. The operator $A$ has a unique extension $A_{-1} \in \mathcal L(X,X_{-1})$ which generates a  $C_0$-semigroup $(T_{-1}(t))_{t\ge 0}$ on $X_{-1}$ which is an extension of $(T(t))_{t\ge 0}$. \\We briefly recall the definitions of Young functions and Orlicz spaces. A function $\Phi:[0,\infty)\rightarrow \mathbb R$ is called a \emph{Young function (or Young function generated by $\varphi$)} if
  \begin{equation*}
    \Phi(t) = \int_0^t \varphi(s) \, ds, \qquad t \geq 0,
  \end{equation*}
where the function $\varphi \colon [0, \infty) \to \mathbb R$  is right-continuous and 
 nondecreasing, $\varphi(0)= 0$, $\varphi(s) >0$ for $s>0$ and
$\lim_{s \to \infty}\varphi(s) = \infty$.
\newline
 Let $I \subset \mathbb R$ be an interval. For a Young function $\Phi$,
 let $L_\Phi(I;U)$ be the set of all equivalence classes (w.r.t.\ equality almost everywhere) of Bochner-measurable functions $u \colon I \to U$ for which there is a $k > 0$ such that
 \begin{equation*}
   \int_I \Phi(k^{-1} \|u(x)\|_{U}) \, dx < \infty.
 \end{equation*}
With the \emph{Luxemburg norm} 
\begin{equation*}
     \|u\|_{\Phi}:= \|u\|_{L_{\Phi}(I;U)}: =\inf\left\{k>0 \Bigm| \int_I \Phi(k^{-1}\| u(x)\|) \,dx \leq 1 \right\},
  \end{equation*}
the space $(L_\Phi (I;U),  \|\cdot\|_\Phi)$ is a Banach space \cite[Thm.~3.9.1]{Kufner}.  If the interval $I \subset \mathbb R$ is bounded, then $L^\infty(I, U)$ is a linear subspace of $L_\Phi(I, U)$. 
\begin{definition}
 For bounded intervals $I \subset \mathbb R$ the space $E_\Phi(I, U)$ is defined as  
\begin{equation*}
    E_\Phi(I, U) = \overline{L^\infty(I, U)}^{\|\cdot\|_{L_{\Phi}(I;U)}}.
  \end{equation*}
  The norm $\|\cdot\|_{E_{\Phi}(I;U)}$ refers to $\|\cdot\|_{L_{\Phi}(I;U)}$.
\end{definition}

 In case $U = \mathbb K$ we write $L_\Phi(I)=L_\Phi (I,\mathbb K)$ and $E_\Phi(I)=E_\Phi (I,\mathbb K)$ for short. The Orlicz spaces generalize the $L^p$ spaces for $1< p< \infty$. More details can be found in \cite{Adams,KrasnRut,Kufner,RaoOrlicz,Zaanen} and also in the appendix of \cite{JNPS16}.

Throughout this paper we use the following convention. By $Z(0,t;U)$ we refer to either a Lebesgue space $L^p(0,t;U)$, with $1\leq p \leq \infty$ or an Orlicz spaces $E_\Phi(0,t;U)$, for some Young function $\Phi$.

\begin{definition}\label{def:iISS}
 We call the  system $\Sigma(A,B)$  {\em (finite-time) admissible with respect to $Z$} (or {\em $Z$-admissible}), if for all $t>0$ and all $u\in Z(0,t;U)$ it holds that
\begin{equation}\label{eq:Admissibility}
	\int_{0}^{t}T_{-1}(s)Bu(s)\ ds \in X.
\end{equation}
\end{definition}
By a \emph{(mild) solution} of \eqref{eqn1} we mean the function defined by the variation of parameters formula
\begin{equation}\label{eqn2}
  x(t)=T(t)x_0+\int_0^t T_{-1}(t-s)B u(s) \, ds, \qquad t \geq 0.
\end{equation}

If $\Sigma(A,B)$ is admissible with respect to $Z$, then all mild solutions of \eqref{eqn1} are $X$-valued and by the closed graph theorem there exists a constant $c(t)$ such that
\begin{equation}\label{eq:adm2}
	\left\|\int_{0}^{t}T_{-1}(s)Bu(s)\ ds\right\| \leq c(t) \|u\|_{Z(0,t;U)}\quad \forall u\in Z(0,t;U).
\end{equation}
Moreover, $\Sigma(A,B)$ is admissible if \eqref{eq:Admissibility} holds for some $t>0$. 

\begin{definition}
We call the  system $\Sigma(A,B)$ {\em infinite-time admissible with respect to $Z$} (or {\em infinite-time $Z$-admissible}), if the system is $Z$-admissible and the optimal constants in \eqref{eq:adm2} satisfy $c_\infty:=\sup_{t>0} c(t)<\infty$.
\end{definition}
  A $C_0$-semigoup $(T(t))_{t \geq 0}$ is called {\em strongly stable}, if $ \lim_{t \to \infty} T(t) x = 0$
holds for all $x \in X$.

\begin{rem}\label{remadm}
Clearly, infinite-time admissibility implies admissibility. Also, if $B$ is a bounded operator from $U$ to $X$, then $\Sigma(A,B)$ is admissible. \newline
If the semigroup $(T(t))_{t\ge 0}$ is exponentially stable, that is, there exist constants $M,\omega>0$ such that
\begin{equation}\label{eqnexp}
\|T(t)\|\le Me^{-\omega t}, \quad t\ge 0,
\end{equation}
then  it is not hard to see that $Z$-infinite-time admissibility is equivalent to $Z$-admis\-sibility \cite[Lem.~8]{JNPS16}. In general, $Z$-admissibility does not imply infinite-time  $Z$-admis\-sibility, not even if $B$ is bounded or if the semigroup is strongly stable, see \cite[Ex.~3.1]{DaM13} for an example with $Z = L^\infty$ or \cite{JacSchn} with $Z= L^2$.
\end{rem}

In the definition below we use the following classes of comparison functions from Lyapunov theory.
\begin{align*} 
\mathcal K ={}& \{\mu \colon \mathbb R_0^+ \rightarrow \mathbb R_0^+ \:|\: \mu(0)=0,\, \mu\text{ continuous and }  \text{strictly increasing}\},\\
 \mathcal K_\infty ={}& \{\theta \in \mathcal K \:|\:  \lim_{x\to\infty} \theta(x)=\infty\},\\
\mathcal{L}={}&\{\gamma \colon \mathbb R_0^+ \rightarrow \mathbb R_0^+  \:|\:\gamma \text{ continuous, } \text{strictly decreasing and } \lim_{t\to\infty}\gamma(t)=0 \}.
\end{align*}

\begin{definition} 
The system $\Sigma(A,B)$  is called {\em strongly input-to-state stable with respect to  $Z$} (or {\em  $Z$-sISS}), if there exist functions $\mu\in \mathcal K$  and $\beta \colon X \times \mathbb R_0^+ \to \mathbb R_0^+$  such that
\begin{enumerate} 
\item $\beta(x, \cdot) \in \mathcal L$ for all $x \in X$, $x \neq 0$ and
\item  for every $t\ge 0$, $x_0\in X$ and  $u\in Z(0,t;U)$ the state  $x(t)$ lies in $X$  and
  \begin{equation}\label{eq:sISS-estimate}
    \left\| x(t)\right\| \le \beta(x_0,t)+  \mu (\|u\|_{Z(0,t;U)}).
  \end{equation}
\end{enumerate} 

The system $\Sigma(A,B)$  is called  {\em strongly integral input-to-state stable with respect to  $Z$} (or {\em $Z$-siISS}), if   there exist  functions $\theta\in \mathcal K_\infty$, $\mu \in {\mathcal K}$ and  $\beta \colon X \times \mathbb R_0^+ \to \mathbb R_0^+$  such that
\begin{enumerate} 
\item $\beta(x, \cdot) \in \mathcal L$ for all $x \in X$, $x \neq 0$ and
\item  for every $t\ge 0$, $x_0\in X$ and  $u\in Z(0,t;U)$ the state $x(t)$ lies in $X$  and
  \begin{equation}
    \label{eq:siISS-estimate}
    \left\| x(t)\right\| \le \beta(x_0,t)+ \theta \left(\int_0^t \mu (\|u(s)\|_{U})ds\right).
  \end{equation}
  For $Z=L^{\infty}$, we will sometimes write siISS instead of $Z$-iISS.
\end{enumerate}
\end{definition}

 \begin{rem}
   The definitions given above generalize the notions of ISS and iISS. It is easy to see that ISS implies sISS and iISS implies siISS.\\
The definition of strong input-to-state stability appeared first in \cite{MiW16}. There the authors have the following additional condition: There is a $\sigma \in \mathcal K_\infty$ such that for all $x \in X$ and $t \geq 0$: \[\beta(x,t) \leq \sigma(\|x\|).\] In our situation of linear systems this condition is redundant. Indeed Proposition \ref{proprel1} below shows that strong ISS implies the strong stability of the semigroup $(T(t))_{t \geq 0}$. By the uniform boundedness principle there is some $M >0$ such that $\|T(t)\|_{\mathcal L(X)} \leq M$. Taking $\sigma(s) = Ms$ yields $\sigma \in \mathcal K_\infty$ and $\beta(x,t) \leq \sigma(\|x\|)$.
 \end{rem}

\begin{proposition}\label{proprel1}
Let $Z$ be either $L^p$, $p\in[1,\infty]$, or $E_\Phi$. Then we have:
\begin{enumerate}[label=(\roman*)]
\item\label{prop:ISSit1} The following are equivalent
\begin{enumerate}
\item \label{prop:ISSit11}$\Sigma(A,B)$ is $Z$-sISS,
\item \label{prop:ISSit12}$\Sigma(A,B)$ is infinite-time $Z$-admissible and $(T(t))_{t\ge 0}$  is strongly stable.
\end{enumerate}
\item \label{prop:ISSit2} If $\Sigma(A,B)$ is  $Z$-siISS,  then the system is $Z$-admissible and $(T(t))_{t\ge 0}$  is strongly stable.
\end{enumerate}
\end{proposition}
\begin{proof} Clearly, $Z$-sISS and $Z$-siISS imply $Z$-ad\-missibi\-lity.\\
If $\Sigma(A,B)$ is $Z$-sISS or $Z$-siISS then, by setting $u=0$, it follows that for all $x \neq 0$ we have $\|T(t)x\| \leq \beta(x,t)$ for all $t$ and hence $\lim_{t \to \infty} T(t)x = 0$ which shows that $(T(t))_{t\ge 0}$  is strongly stable. This shows \ref{prop:ISSit2}. In the case that $\Sigma(A,B)$ is $Z$-sISS we get 
\begin{equation*}
  \begin{split}
    \left \Vert \int_0^{t} T_{-1}(s) Bu(s) \,ds \right \Vert &= \left \Vert \int_0^{t} T_{-1}(s) B\frac{u(s)}{\|u\|_{Z(0,t;U)}} \,ds \right \Vert \|u\|_{Z(0,t;U)}\\ &\leq  \mu(1)\|u\|_{Z(0,t;U)}. 
  \end{split}
\end{equation*}
for any element of $Z(0,t;U)$, $u \neq 0$. This shows that $\Sigma(A,B)$ is infinite-time $Z$-admissible, and thus (a)$\Rightarrow$(b) in \ref{prop:ISSit1}.
\\
Conversely, if the system $\Sigma(A,B)$ is $Z$-infinite-time-admissible and  $(T(t))_{t\ge 0}$  is strongly stable we set $\beta(x,t)= \|T(t)x\|$, $\sigma(s):= Ms$, where $M:= \sup_{t\geq 0}\|T(t)\|$, and $\mu(s) = c_\infty s$. Then $\sigma$ belongs to $\mathcal K_\infty$, $\beta(x,t)\leq M \|x\|$  and $\|x(t)\|\leq \beta(x_0,t) + \mu( \|u\|_{Z(0,t;U)})$ for all $t\geq0$, $x \in X$.
\end{proof}

We remark that Proposition \ref{proprel1}  can be proved for more general function spaces $Z$ with properties as discussed in  \cite{JNPS16}. 
The (short) proof of the following proposition follows the same lines as in \cite[Prop.~2.10]{JNPS16}.
\begin{proposition}\label{proprel2a}
  Let $p\in[1,\infty)$. If the system $\Sigma(A,B)$ is $L^p$-sISS then it is $L^p$-siISS.
\end{proposition}

\section{Main results}
\label{sec:3}
In this section we study the relation between strong integral ISS with respect to $L^\infty$ and (infinite-time) admissibility with respect to some Orlicz. We need two technical lemmata.

\begin{lemma}\label{lem:Phi1}
Let $\Phi$ be a Young function. Then there exists some Young function $\Phi_1$ such that $\Phi \leq \Phi_1$ and $\sup_{x>0}\Phi(cx)/ \Phi_{1}(x) <\infty$  for all $c >0$.
\end{lemma}

\begin{proof}
  Let the Young function $\Phi\colon [0,\infty) \to \mathbb R$ be generated $\varphi$, i.e.\ $\Phi(x) = \int_0^x \varphi(t) \, dt$. We define two Young functions $\Lambda, \Psi \colon [0,\infty) \to \mathbb R$ by
  \begin{equation*}
    \Lambda (x) = \int_0^x \varphi(\sqrt{t}) \,dt
  \end{equation*}
and $\Psi(x) = \Phi(x^2)$. Then, obviously, $\Phi \leq \Lambda$ holds on the interval $[0,1]$ and  $\Phi \leq \Psi$ holds on $[1,\infty)$. Hence $\Phi_1 \colon [0,\infty) \to \mathbb R$,
\begin{equation*}
  \Phi_1(x) = \begin{cases}
    \Lambda(x), &x< 1, \\
    \frac{\Lambda(1)}{\Psi(1)} \Psi(x), &x\geq 1,
  \end{cases}
\end{equation*}
defines a Young function with $\Phi \leq \Phi_1$ (Note that $\Lambda(1) \geq \Psi(1) = \Phi(1)$).\\
We now show that for each $c >0$ the function $x \mapsto \Phi(cx)/\Phi_1(x)$ is bounded on $(0,\infty)$. For $0 < c \leq 1$ this  simply follows from the monotonicity of $\Phi$. Indeed we have
\begin{equation*}
  \frac{\Phi(cx)}{\Phi_1(x)} \leq \frac{\Phi(x)}{\Phi_1(x)} \leq 1 
\end{equation*}
for all $x >0$.\\
Now let $c >1$. For $x \geq c$ we have that
\begin{equation*}
  \frac{\Phi(cx)}{\Phi_1(x)} = \frac{\Psi(1)\Phi(cx)}{\Lambda(1)\Phi(x^{2})} \leq \frac{\Psi(1)}{\Lambda(1)}.
\end{equation*}
 For an arbitrary Young function $\Omega$, generated by $\omega$, we have for all $y >0$ 
\begin{equation*}
  \frac{\Omega(y)}{y} = \frac{1}{y} \int_0^y \omega(t) \,dt \geq  \frac{1}{y} \int_{y/2}^y \omega(t) \,dt \geq \frac{1}{2} \omega\left(\frac{y}{2} \right)
\end{equation*}
and 
\begin{equation*}
  \frac{\Omega(y)}{y} = \frac{1}{y} \int_0^y \omega(t) \,dt \leq \omega(y).
\end{equation*}
Therefore we have
\begin{equation*}
  \frac{\Phi(cx)}{\Lambda(x)} = c  \frac{\Phi(cx)}{cx} \frac{x}{\Lambda(x)} \leq 2c \frac{\varphi(cx)}{\varphi\left(\sqrt{\frac{x}{2}} \right)} \leq 2c,
\end{equation*}
where the last inequality holds for all $x\in (0, 1/(2c^2)]$.
Since the continuous function $x \mapsto \Phi(cx)/\Phi_1(x)$ is bounded on the compact interval $[1/(2c^2), c]$, the claim follows.
\end{proof}

\begin{lemma}\label{lem:MoTr}[Lemma 8.1 in \cite{MoTr50}]
  Let $I \subset \mathbb R$ be an interval. Let $(u_n)_{n \in N}$ be sequence in  $(u_n)_{n \in N} \subset L_\Phi(I)$ such that for all $r >0$ the sequence $(r  u_n)_{n\in \mathbb N}$ is mean convergent to zero, i.e. $\lim_{n \to \infty} \int_I \Phi(r u_n(x)) \, dx = 0$. Then we have $\lim_{n \to \infty}\|u_n\|_{\Phi} =0$.
\end{lemma}
Now we are ready to prove a sufficient condition for a system $\Sigma(A,B)$ to be siISS. The proof is a careful refinement of the technique used in the proof of \cite[Thm.~3.1]{JNPS16} -- the situation there being easier as Lemma \ref{lem:MoTr} is not needed.
\begin{theorem}\label{thm:main2b}
Suppose there is a Young function $\Phi$ such that the system $\Sigma(A,B)$ is $E_{\Phi}$-sISS. Then the system $\Sigma(A,B)$ is $L^{\infty}$-siISS.
\end{theorem}
\begin{proof}
  Let $\Phi_1$ be the Young function given by Lemma \ref{lem:Phi1}. We define $\theta \colon [0,\infty) \to [0,\infty)$ by $\theta(0) = 0$ and 
\begin{multline*}
  \theta (\alpha) = \sup\left\{ \left\Vert \int_0^t T_{-1}(s) B u(s) \,ds  \right\Vert \Bigm| \right. u \in L^\infty (0,t;U),\,t \geq 0,\\ \left.\int_0^t \Phi_{1}(\|u(s)\|_U) \,ds \leq \alpha \right\},
\end{multline*}
for $\alpha >0$. The function $\theta$ is well-defined, since by infinite-time admissibility, \cite[ Remark 39]{JNPS16} and  $\Phi \leq \Phi_1$ we have for $t\geq 0$, $u \in L^\infty(0,t;U)$ 
\begin{equation*}
  \begin{split}
    \left\| \int_0^t T_{-1}(s) B u(s) \,ds \right \|&\leq c_\infty \|u\|_{E_\Phi (0,t;U)} \\ &\leq c_\infty \left(1+ \int_0^t \Phi (\|u(s)\|_U) \,ds \right)\\ &\leq c_\infty \left(1+ \int_0^t \Phi_{1}(\|u(s)\|_U) \,ds \right)\\ &\leq c_\infty(1+\alpha).
  \end{split}
\end{equation*}

Clearly, $\theta$ is non-decreasing.
If we can show that $\lim_{t \searrow 0}\theta(t)=0$, then, by  \cite[Lemma 2.5]{CLS98},  there exists  $\tilde{\theta}\in \mathcal{K}_\infty$ with $\theta\leq\tilde{\theta}$. Since $\Phi_{1} \colon [0, \infty) \to [0, \infty)$ is a Young function, $\Phi_{1} \in \mathcal K_{\infty}$. The definition of $\theta$ yields that
 \[\left\Vert \int_0^t T_{-1}(s) B u(s) \,ds  \right\Vert \leq \theta\left(\int_0^t \Phi_{1}(\|u(s)\|_U) \,ds \right) \leq\tilde{\theta}\left(\int_0^t \Phi_{1}(\|u(s)\|_U) \,ds \right)\]
 for all  $u \in L^\infty(0,t;U)$ which means that $\Sigma(A,B)$ is siISS.

To show $\lim_{t \searrow 0}\theta(t)=0$, let $(\alpha_{n})_{n \in \mathbb N}$ be a sequence of positive real numbers converging to $0$. By the definition of $\theta$, for any $n\in\mathbb N$ there exist a $u_{n}\in L^{\infty}(0,\infty;U)$ with compact essential support such that
\begin{equation*}
  \int_0^{\infty} \Phi_1(\|u_n(s)\|_U) \,ds < \alpha_n
\end{equation*}
and 
\begin{equation}\label{thm31:eq5b}
  \left|\theta(\alpha_{n})- \left \Vert \int_{0}^{\infty}T_{-1}(s)Bu_{n}(s) \,ds\right \Vert \right|<\frac{1}{n}.
\end{equation}

It follows that the sequence $(\|u_{n}(\cdot)\|_U)_{n \in \mathbb N}$ is $\Phi_{1}$-mean convergent to zero.
Hence, for all $r >0$ the sequence $(r \|u_{n}(\cdot)\|_U)_{n \in \mathbb N}$ is $\Phi$-mean convergent to zero. By Lemma \ref{lem:MoTr} the sequence converges to zero with respect to the norm of the space $L_\Phi(0,\infty)$ and hence $\lim_{n \to \infty} \|{u}_{n}\|_{L_{\Phi}(0,\infty;U)}=0$. Therefore, by admissibility,
\[  \left \Vert \int_0^{\infty} T_{-1}(s) B  u_{n}(s) \, ds \right \Vert \leq c_\infty \| u_n\|_{L_\Phi(0,\infty;U)} \to 0, \] 
as $n\to \infty$. Together with \ref{thm31:eq5b} we obtain that$\lim_{n \to \infty} \theta(\alpha_n) =0$.
\end{proof}
We omit the proof of the following Lemma as it is implicitly given in the proof of \cite[Lem.~8]{JNPS16}. Note that here both assumption and conclusion are weaker.

\begin{lemma}\label{lem:3.24b}
  Let $(T(t))_{t\ge 0}$ be a semigroup and let $\Sigma(A,B)$ be siISS. Then there exist $\tilde\theta, \Phi \in \mathcal K_\infty$ such that $\Phi$ is a Young function which is continuously differentiable on $(0,\infty)$ and for all $u \in L^\infty (0,1;U)$,
  \begin{equation} 
       \left \Vert \int_0^1 T_{-1} (s) B u(s) \, ds \right \Vert \leq \tilde\theta \left(\int_0^1 \Phi (\Vert u(s) \Vert_U )\, ds \right).
  \end{equation}
\end{lemma}

\begin{theorem}\label{thm:main2c}
  Assume that the system $\Sigma(A,B)$ is siISS. Then there is a Young function $\Phi$ such that the system $\Sigma(A,B)$ is $E_\Phi$-admissible. If, additionally, function $\mu$ in \eqref{eq:siISS-estimate} can be chosen as a Young function, then 
 $E_{\mu}$-sISS.
\end{theorem}

\begin{proof}
Let $\Phi$ be the Young function given by  Lemma \ref{lem:3.24b}.  Analogous as in the proof of (ii)$\Rightarrow$(i) in \cite[Theorem 3.1]{JNPS16},  but using Lemma \ref{lem:3.24b} instead of \cite[Lemma 3.5]{JNPS16}, we deduce that $\int_{0}^{1}T_{-1}(s)Bu(s)ds\in X$ for all $u\in E_{\Phi}(0,1;U)$.

Now assume that the function $\mu$ in \eqref{eq:siISS-estimate} is a Young function. The admissiblility with respect to $E_\mu$ is now easier to see: For $u \in E_\mu(0,t;U)$ we pick a sequence $(u_n)_{n\in \mathbb N} \subset L^\infty(0,t;U)$ such that $\lim_{n \to \infty}\|u_n-u\|_{E_\mu(0,t;U)}$ and $\|u_n-u_m\|_{E_\mu(0,t;U)}\leq 1$ for all $m,n \in \mathbb N$. Then the siISS-estimate and Lemma 3.8.4 (i) in \cite{Kufner} yield
\begin{equation*}
  \begin{split}
     \left \Vert \int_0^t T_{-1}(s) B (u_n(s)-u_m(s)) \,ds \right \Vert & \leq \theta \left(\int_0^t \mu( \|u_n(s)-u_m(s)\|_U) \, ds \right) \\ & \leq  \theta \left(\|u_n-u_m\|_{E_{\mu}(0,t;U)}  \right).
  \end{split}
\end{equation*}
Hence $(\int_0^t T_{-1}(s) B u_n(s) \,ds)_{n \in \mathbb N}$ is a Cauchy sequence in $X$ and the same argument as above shows that $\int_0^t T_{-1}(s) B u(s) \,ds \in X$ holds. For all $t \geq 0$, $u \in E_\mu(0,t;U)$, $u \neq 0$, we have by Lemma 3.8.4 in \cite{Kufner}
\begin{equation*}
  \begin{split}
      \left\Vert \int_0^{t} T_{-1}(s) B u(s) \,ds  \right\Vert &= \left\Vert \int_0^{t} T_{-1}(s) B \frac{u(s)}{\|u\|_{ E_\mu(0,t;U)}} \,ds \right\Vert \|u\|_{ E_\mu(0,t;U)}\\ & \leq \theta \left(\int_0^{t} \mu \left(\frac{\|u(s)\|_{U}}{\|u\|_{ E_\mu(0,t;U)}}\right) \,ds \right)\|u\|_{ E_\mu(0,t;U)} \\ &\leq \theta(1)\|u\|_{ E_\mu(0,t;U)}.
  \end{split}
\end{equation*}
Hence the system $\Sigma(A,B)$ is infinite-time $E_\mu$-admissible.
\end{proof}
 
 It is well-known that for unbounded intervals $I \subset \mathbb R$ there exist bounded functions in $ L^p(I)$, $p > 1$, which do not belong $L^1(I)$. The following Lemma is an Orlicz-space version of that result.

\begin{lemma}\label{lem:notL1}
  Let $I \subset \mathbb R$ an unbounded interval. Then for each Young function $\Phi$ there exists a strictly positive function $u_0 \in L_\Phi(I) \cap  L^\infty(I)$ with $u_0 \notin L^1(I)$. 
\end{lemma}

\begin{proof}
For any Young function $\Phi$ holds $\lim_{t \to 0} \Phi(t)/t = 0$. Hence there is a sequence $(t_k)_{k \in \mathbb N} \subset (0,1)$ such that for all $k \in \mathbb N$ we have
    \begin{equation*}
      \frac{\Phi(t_k)}{t_k} \leq 2^{-k}.
    \end{equation*}
Since $I$ is unbounded there is a sequence $(I_k)_{k \in \mathbb N}$ of measurable disjoint sets $I_k \subset \mathbb R$ with \[I = \bigcup_{k\in \mathbb N} I_k\] and $\lambda(I_k) = t_k^{-1}$. 
We define $u_0 \colon I \to \mathbb R$ by $u_0 = \sum_{k \in \mathbb N} t_k \chi_{I_k}$. 
Then $u_0 \in L^\infty(I)$. Further we have
\begin{equation*}
  \int_I |u_0(x)| \,dx = \sum_{k= 0}^\infty t_k \lambda(I_k) = \sum_{k= 0}^\infty 1 = \infty
\end{equation*}
and 
\begin{equation*}
  \int_\Omega \Phi(|u_0(x)|) \,dx = \sum_{k= 0}^\infty \Phi(t_k) \lambda(I_k) \leq \sum_{k= 0}^\infty 2^{-k} = 2.
\end{equation*}
Hence we have $u_0 \notin L^1(I)$ and $u \in L_\Phi(I)$.
\end{proof}
The following Lemma is an integral version of the well-known fact that there is no series which diverges less rapidly than any other \cite[p.~299]{Knopp}.
\begin{lemma}
  \label{lem:funct-h}
For all $f \in L^\infty(0,\infty)\setminus L^1(0,\infty)$, $f > 0$, there exists a continuously differentiable decreasing function $h \colon (0,\infty) \to [0,\infty)$ such that $\lim_{t\to \infty} h(t) = 0$ and $\int_0^\infty h(s) f(s) \,ds = \infty$.
\end{lemma}
\begin{proof}
  For $n \in \mathbb N$ let $c_n= \int_n^{n+1} f(s) \,ds$. Then we have $\sum_{n =0}^\infty c_n = \infty$ and  by \cite[p.~299]{Knopp} the series $\sum_{n =0}^\infty c_{n}d_n$ is also divergent, where $d_n:= (\sum_{k= 0}^n c_n)^{-1}$. Since the function $f$ is positiv, the sequence $(d_n)_{n \in \mathbb N}$ is strictly decreasing and hence there exists a continuously differentiable decreasing function $h \colon (0,\infty) \to [0,\infty)$ such that $d_{n+1} \leq h|_{[n,n+1]} \leq d_{n}$ for all $n \in \mathbb N$. It is clear that $h(t)\searrow0$ as $t\to\infty$ and $\int_0^\infty h(s) f(s) \,ds = \infty$.
\end{proof}

The following theorem shows that $E_\Phi$-infinite-time admissibility and $L^\infty$-strong integral input-to-state stability are not equivalent, i.e.\ we cannot drop the Young function condition in the second part of Theorem \ref{thm:main2c}. Note that, by \cite[Theorem 15]{JNPS16}, a linear system $\Sigma(A,B)$ is $L^\infty$-iISS if and only if it is infinite-time admissible with respect to $E_\Phi$ for some Young function $\Phi$. In contrast, Theorem \ref{thm:main2b} and the following result show that without the exponential stability of the semigroup, sISS with respect to $E_{\Phi}$ is a stronger notion than siISS. 

\begin{theorem}\label{thm:counterex}
There is a system $\Sigma(A,B)$ such that
\begin{enumerate}
\item $\Sigma(A,B)$ is infinite-time admissible with respect to $L^1$, in particular $\Sigma(A,B)$ is siISS with respect to $L^1$ and hence siISS.
\item  $\Sigma(A,B)$ is not $E_\Phi$-sISS for any Young function $\Phi$.
\end{enumerate}
Moreover $\Sigma(A,B)$ is not infinite-time admissible with respect to $L^\infty$. In particular siISS does not imply $L^\infty$-sISS.
\end{theorem}

\begin{proof}
Let $(T(t))_{t \geq 0}$ be the left-translation semigroup on $X= L^1(0,\infty)$, i.e. $(T(t)f)(s) = f(t+s)$, $f \in X$, which  is strongly stable. The generator is given by
  \begin{equation*}
   Af:= f', \qquad  D(A) = \{f \in L^1(0,\infty) \mid f \in AC(0,\infty) \text{ and } f' \in L^1(0,\infty)\},
  \end{equation*}
see e.g. \cite{engelnagel}. We choose $U= X=  L^1(0,\infty)$ as input space and $B = I$ as control operator. 
System $\Sigma(A,B)$ is infinite-time $L^{1}$-admissible  because for any $u \in L^1(0,t;X)$ we have 
    \begin{equation*}
      \begin{split}
        \left \|\int_0^t T(s) B u(s) \,ds  \right \|_X & \leq  \int_0^t \|T(s) B u(s)\|_{X} \,ds \\ 
        &\leq \int_0^t  \| u(s)\|_{X}  \,ds\\ 
        &= \|u\|_{L^1(0,t;L^{1}(0,\infty))}.
      \end{split}
    \end{equation*}
Now let us fix a Young function $\Phi$.
 In order to show that $\Sigma(A,B)$ is not $E_{\Phi}$-infinite-time admissible, we construct a function $u$ in the following way. Let $u_0\in L_{\Phi}(0,\infty)\cap L^{\infty}(0,\infty)$ be given by Lemma \ref{lem:notL1} (with $I=(0,\infty)$) and let $h$ be given by Lemma \ref{lem:funct-h} applied to $f:= u_0$. Now set $g=-h'$ and define $u \colon (0,\infty) \to L^1 (0,\infty)$,
  \begin{equation*}
    [u(s)](r) = g(r) \chi_{[s,\infty)}(r) u_0(s),
  \end{equation*}
  which is well-defined since for $s\in(0,\infty)$, $\int_s^\infty |g(r)| \,dr=h(s)$ and
\begin{equation*}
  \|u\|_{L^{1}(0,t;X)}=\int_0^t u_0(s) \int_s^\infty |g(r)| \,dr \, ds = \int_{0}^{t}u_{0}(s)h(s)\, ds.
\end{equation*}
Hence, the restriction of $u$ to the interval $[0,t]$ belongs to $L^1(0,t;L^1(0,\infty))$ for all $t\geq 0$ but $u \notin L^1(0,\infty;L^1(0,\infty))$. We obtain using that $[u(s)](r) \geq 0$ for all $r,s >0$ 
and  $[u(s)](r) = 0$ for all $r \in [0,s)$ together with Fubini's theorem,
\begin{equation*}
   \left \|\int_0^t T(s) B u(s) \,ds  \right \|_X = \|u\|_{L^1(0,t;X)}.
\end{equation*}
Since $u_0 \in L^\infty(0,\infty)$ and for all $s > 0$ 
\begin{equation}\label{eq:inequ}
  \begin{split}
    \|u(s)\|_X = \int_0^\infty [u(s)](r) \,dr = u_0(s) \int_s^\infty g(r) \,dr \leq u_0(s) h(s)
  \end{split}
\end{equation}
we have that $u \in L^\infty(0,\infty;X)$ and
\begin{equation}\label{eq:u-in-L-infty}
  \|u\|_{L^\infty(0,\infty;X)} \leq \|u_0\|_{L^\infty(0,\infty)}\|h\|_{L^\infty(0,\infty)}.
\end{equation}
Therefore, $u|_{[0,t]} \in E_\Phi(0,t;X)$ and by \eqref{eq:inequ} follows that  
\begin{equation}\label{eq:u-in-E-Phi}
  \|u\|_{E_\Phi(0,t;X)} \leq \|h\|_{L^{\infty}(0,\infty)} \|u_0\|_{E_\Phi(0,t)} \leq \|h\|_{L^{\infty}(0,\infty)} \|u_0\|_{L_\Phi(0,\infty)}. 
\end{equation}
If $\Sigma(A,B)$ was infinite-time $E_{\Phi}$-admissibility, \eqref{eq:u-in-E-Phi} would lead to
\begin{equation*}
   \|u\|_{L^1(0,t;X)}=\left \|\int_0^t T(s) B u(s) \,ds  \right \|_X \leq c_\infty \|u\|_{E_\Phi(0,t;X)} \leq c_{\infty} \|g\|_{\infty} \|u_0\|_{L_\Phi(0,\infty)}
\end{equation*}
for some $c_{\infty}>0$ independent of $u$ and $t$. Letting $t\to\infty$, this gives a contradiction as $\|u\|_{L^1(0,t;X)}$ tends to $\infty$. 
Using \eqref{eq:u-in-L-infty} instead of \eqref{eq:u-in-E-Phi} we similarly can similarly derive a contradiction for infinite-time $L^\infty$-admissibility.
\end{proof}

A Young function $\Phi$ \textit{satisfies the $\Delta_{2}$-condition} if there exist a $k>0$ and $s_0 \geq 0$ such that 
$\Phi(2s)\le k \Phi(s)$ for all $s \geq 0$. The following result complements \cite[Thm.~3.2]{JNPS16}.

\begin{theorem}
Let $\Phi$ be a Young function which satisfies the $\Delta_2$-condition with $s_0=0$. If the system $\Sigma(A,B)$ is $E_\Phi$-sISS then it is  $E_\Phi$-siISS.
\end{theorem}

\begin{proof}
Similarly to the proof of Theorem \ref{thm:main2b}, we consider a nondecreasing function $\theta \colon [0,\infty) \to [0,\infty)$ defined by $\theta(0) = 0$ and 
\begin{multline*}
  \theta (\alpha) = \sup\left\{ \left\Vert \int_0^t T_{-1}(s) B u(s) \,ds  \right\Vert \Bigm| u \in E_\Phi (0,t;U),~t \geq 0,\right.\\ \left. \int_0^t \Phi(\|u(s)\|_U) \,ds \leq \alpha \right\},
\end{multline*}
for $\alpha >0$. It follows as in Theorem \ref{thm:main2b} that $\theta$ is well-defined and non-decreasing. As in the proof of \ref{thm:main2b}, it remains to show that $\theta$ is continuous in 0. This follows from the $\Delta_2$-condition. Indeed, let $(\alpha_{n})_{n \in \mathbb N}$ be a sequence of positive real numbers converging to $0$. By the definition of $\theta$, for any $n \in \mathbb N$ there exist $t_n \geq 0$ and $u_{n}\in L_\Phi(0,t_n;U)$ such that 
\begin{equation*}
  \int_0^{t_n} \Phi(\|u_n(s)\|_U) \,ds < \alpha_n
\end{equation*}
and 
\begin{equation*}
  \left|\theta(\alpha_{n})- \left \Vert \int_{0}^{t_n}T_{-1}(s)Bu_{n}(s) \,ds\right \Vert \right|<\frac{1}{n}.
\end{equation*}
By extending the functions $u_n$ to $[0,\infty)$ by $0$, we can assume that $(u_n)_{n \in \mathbb N} \subset  L_\Phi (0,\infty;U)$ and both estimates above hold with $t_n= \infty$.
It follows that the sequence $(\|u_{n}(\cdot)\|)_{n \in \mathbb N}$ is $\Phi$-mean convergent to zero. Hence, by Lemma 3.10.4 in \cite{Kufner} it converges to zero in $L_\Phi(0,\infty; U)$, for $\Phi$ satisfies the $\Delta_2$-condition. 
By $E_\Phi$-infinite-time admissibility we conclude that $\lim_{n \to \infty} \theta(\alpha_n) = 0$.
\end{proof}

\section{Concluding remarks}
We would like to remark that for our results the strong stability of the semigroup generated by $A$ is not really needed. Indeed we could replace this condition by boundedness of the semigroup and the sISS estimate or, respectively, the siISS estimate with initial value zero. Note that this differs from the situation of exponentially stable semigroups in \cite{JNPS16}, where finite-time admissibility is equivalent to infinite-time admissibility because of exponential stability.\\
 As explained before, until now it is not clear whether $L^{\infty}$-ISS and iISS are equivalent for linear systems. In the strong setting, we have seen that this does not hold as Theorem \ref{thm:counterex} shows that the implication siISS $\implies$ $L^{\infty}$-sISS fails in general. This behavior is different from the ISS case. We believe that also the other implication, $L^{\infty}$-sISS $\implies$ siISS, fails in general. Constructing a counterexample is subject to future work.

\begin{thebibliography}{10}

\bibitem{Adams}
R.~A. Adams.
\newblock {\em {Sobolev spaces}}.
\newblock Academic Press, New York-London, 1975.
\newblock {P}ure and Applied Mathematics, Vol. 65.

\bibitem{CLS98}
F.~H. Clarke, Yu.~S. Ledyaev, and R.~J. Stern.
\newblock Asymptotic stability and smooth {L}yapunov functions.
\newblock {\em J. Differential Equations}, 149(1):69--114, 1998.

\bibitem{DaM13}
S.~Dashkovskiy and A.~Mironchenko.
\newblock {Input-to-state stability of infinite-dimensional control systems}.
\newblock {\em Mathematics of Control, Signals, and Systems}, 25(1):1--35,
  August 2013.

\bibitem{engelnagel}
K.-J. Engel and R.~Nagel.
\newblock {\em One-parameter semigroups for linear evolution equations}, volume
  194 of {\em Graduate Texts in Mathematics}.
\newblock Springer-Verlag, New York, 2000.

\bibitem{JNPS16}
B.~Jacob, R.~Nabiullin, J.~R. Partington, and F.~Schwenninger.
\newblock {Infinite-dimensional input-to-state stability and Orlicz spaces}.
\newblock {\em Preprint available at arXiv:1609.09741}, 2016.

\bibitem{JacSchn}
B.~Jacob and R.~Schnaubelt.
\newblock Observability of polynomially stable systems.
\newblock {\em Systems Control Lett.}, 56(4):277--284, 2007.

\bibitem{JacSchwZwa}
B.~Jacob, F.~L. Schwenninger, and H.~Zwart.
\newblock ${L}^{\infty}$-admissibility and ${H}^{\infty}$-calculus.
\newblock In preparation, 2017.

\bibitem{JaLoRy08}
B.~Jayawardhana, H.~Logemann, and E.~P. Ryan.
\newblock {Infinite-dimensional feedback systems: the circle criterion and
  input-to-state stability}.
\newblock {\em Commun. Inf. Syst.}, 8(4):413--414, 2008.

\bibitem{KK2016IEEETac}
I.~Karafyllis and M.~Krstic.
\newblock {ISS} with respect to boundary disturbances for 1-{D} parabolic
  {PDE}s.
\newblock {\em IEEE Trans. Automat. Control}, 61:3712--3724, 2016.

\bibitem{Krstic16}
I.~Karafyllis and M.~Krstic.
\newblock {ISS} in different norms for 1-{D} parabolic {PDE}s with boundary
  disturbances.
\newblock {\em SIAM Journal on Control and Optimization}, 55:1716--1751, 2017.

\bibitem{Knopp}
K.~Knopp.
\newblock {Theory and application of infinite series.}
\newblock {London: Blackie \& Sons. XII, 571 pp.}, 1928.

\bibitem{KrasnRut}
M.~A. Krasnosel{\cprime}ski{\u\i} and Ya.~B. Ruticki{\u\i}.
\newblock {\em {Convex functions and {O}rlicz spaces}}.
\newblock {Translated from the first Russian edition by Leo F. Boron}. P.
  Noordhoff Ltd., Groningen, 1961.

\bibitem{Kufner}
A.~Kufner, O.~John, and S.~Fu\v{c}{\'\i}k.
\newblock {\em {Function spaces}}.
\newblock Noordhoff International Publishing, Leyden; Academia, Prague, 1977.
\newblock Monographs and Textbooks on Mechanics of Solids and Fluids;
  Mechanics: Analysis.

\bibitem{MaPr11}
F.~Mazenc and C.~Prieur.
\newblock Strict {L}yapunov functions for semilinear parabolic partial
  differential equations.
\newblock {\em Math. Control Relat. Fields}, 1(2):231--250, 2011.

\bibitem{MiI14b}
A.~Mironchenko and H.~Ito.
\newblock {Characterizations of integral input-to-state stability for bilinear
  systems in infinite dimensions}.
\newblock {\em Mathematical Control and Related Fields}, 6(3):447--466, 2016.

\bibitem{MiW16}
A.~Mironchenko and F.~Wirth.
\newblock {Restatements of input-to-state stability in infinite dimensions:
  what goes wrong}.
\newblock In {\em {Proc. of 22th International Symposium on Mathematical Theory
  of Systems and Networks (MTNS 2016)}}, 2016.

\bibitem{MiW17b}
A.~Mironchenko and F.~Wirth.
\newblock Characterizations of input-to-state stability for
  infinite-dimensional systems.
\newblock {\em Accepted to IEEE Transactions on Automatic Control}, 2017.

\bibitem{MoTr50}
M.~Morse and W.~Transue.
\newblock Functionals {$F$} bilinear over the product {$A\times B$} of two
  pseudo-normed vector spaces. {II}. {A}dmissible spaces {$A$}.
\newblock {\em Ann. of Math. (2)}, 51:576--614, 1950.

\bibitem{RaoOrlicz}
M.~M. Rao and Z.~D. Ren.
\newblock {\em {Theory of {O}rlicz spaces}}, volume 146 of {\em {Monographs and
  Textbooks in Pure and Applied Mathematics}}.
\newblock Marcel Dekker, Inc., New York, 1991.

\bibitem{Sontag89ISS}
E.~D. Sontag.
\newblock {Smooth stabilization implies coprime factorization}.
\newblock {\em IEEE Trans. Automat. Control}, 34(4):435--443, 1989.

\bibitem{Sontag98}
E.~D. Sontag.
\newblock {Comments on integral variants of {ISS}}.
\newblock {\em Systems Control Lett.}, 34(1-2):93--100, 1998.

\bibitem{sontagISS}
E.~D. Sontag.
\newblock {Input to state stability: basic concepts and results.}
\newblock In {\em {Nonlinear and optimal control theory}}, volume 1932 of {\em
  {Lecture Notes in Math.}}, pages 163--220. Springer Berlin, 2008.

\bibitem{Weiss89ii}
G.~Weiss.
\newblock {Admissibility of unbounded control operators}.
\newblock {\em SIAM J. Control Optim.}, 27(3):527--545, 1989.

\bibitem{Zaanen}
A.~C. Zaanen.
\newblock {\em Linear analysis. {M}easure and integral, {B}anach and {H}ilbert
  space, linear integral equations}.
\newblock Interscience Publishers Inc., New York; North-Holland Publishing Co.,
  Amsterdam; P. Noordhoff N.V., Groningen, 1953.

\end{thebibliography}

\def\cprime{$'$}

\end{document}